\documentclass[12pt]{article}
\usepackage{amssymb}
\usepackage{amsmath,amsthm}
\usepackage[latin1]{inputenc}
\usepackage{graphicx}
\usepackage{hyperref}
\usepackage{enumerate}
\usepackage{tikz}
\usepackage{tkz-graph}
\usepackage{mathrsfs}
\usepackage{verbatim}

\hypersetup{colorlinks=true, linkcolor=blue, citecolor=blue, urlcolor=blue}

\newtheorem{remark}{Remark}

\newtheorem{lemma}[remark]{Lemma}
\newtheorem{theorem}[remark]{Theorem}
\newtheorem{proposition}[remark]{Proposition}
\newtheorem{corollary}[remark]{Corollary}

\title{On the super domination number  of  lexicographic product graphs}

\author{M. Dettlaff$^{1}$,  M. Lema\'{n}ska$^{1}$, J. A. Rodr\'{\i}guez-Vel\'{a}zquez$^{2}$,
 R. Zuazua$^{3}$
\\
\\
$^1${\small Department of Technical Physics and Applied Mathematics} \\ {\small Gdansk University of Technology, ul. Narutowicza 11/12 80-233 Gdansk, Poland }\\ {\small magda\@@mifgate.mif.pg.gda.pl}
\\
$^2${\small Departament d'Enginyeria Inform\`atica i Matem\`atiques
}\\
{\small Universitat Rovira i Virgili,  Av. Pa\"{\i}sos Catalans 26,
43007 Tarragona, Spain.} \\{\small juanalberto.rodriguez\@@urv.cat}
\\
$^3${\small Departamento de Matem\'{a}ticas, Universidad Nacional Aut\'{o}noma de M\'{e}xico} \\{\small Ciudad Universitaria, 04510 Coyoacan, Mexico DF, Mexico} \\{\small ritazuazua\@@ciencias.unam.mx}
}

\begin{document}
\maketitle

\begin{abstract}
The   neighbourhood of a vertex $v$ of a graph $G$ is the set $N(v)$ of all vertices 
adjacent to $v$ in $G$.  For $D\subseteq V(G)$ we define $\overline{D}=V(G)\setminus D$. A  set  $D\subseteq V(G)$ is  called  a  super  dominating  set   if for every vertex  $u\in \overline{D}$,  there exists $v\in D$ such that  $N(v)\cap \overline{D}=\{u\}$. The   super domination number  of $G$ is the minimum cardinality among all super dominating sets in $G$.
In this  article we obtain closed formulas and tight bounds for the super dominating number of lexicographic product graphs in terms of invariants of the factor graphs involved in the product. As a consequence of the study, we show that the problem of finding the super domination number  of a graph is NP-Hard.
\end{abstract}



\section{Introduction}
\label{sectIntro}

The {\it neighbourhood} of a vertex $v$ of a graph $G$ is the set $N(v)$ of all vertices 
adjacent to $v$ in $G$.  For $D\subseteq V(G)$ we define $\overline{D}=V(G)\setminus D$. A set $D\subseteq V(G)$  is {\it dominating} in $G$ if every vertex in $\overline{D}$ 
has at
least one neighbour in $D$, \textit{i.e}., $N(u)\cap D\ne \emptyset$ for every $u\in \overline{D}$. The {\it domination number} of $G$, denoted by  
$\gamma (G)$, is the minimum cardinality among all dominating sets in $G$. A dominating set of cardinality $\gamma(G)$ is called a $\gamma(G)$-set. The reader is referred to the books \cite{Haynes1998a,Haynes1998}
 for details on domination in graphs.

The study of super domination in graphs was introduced in \cite{MR3396565}. 
 A  set  $D\subseteq V(G)$ is  called  a  \textit{super  dominating  set}   if for every vertex  $u\in \overline{D}$,  there exists $v\in D$ such that  
\begin{equation}\label{DefinitionPrivateNeighbour}
N(v)\cap \overline{D}=\{u\}.
\end{equation}
If $u$ and $v$ satisfy \eqref{DefinitionPrivateNeighbour}, then we say that $v$ is an \textit{external  private neighbour of $u$ with respect to} $\overline{D}$. The {\it super domination number} of $G$,  denoted by $\gamma_{\rm sp} (G)$, is the minimum cardinality among all super dominating sets in $G$.  A super dominating set of cardinality $\gamma_{\rm sp}(G)$ is called a $\gamma_{\rm sp}(G)$-set.  

In this paper we develop the theory of super domination in lexicographic product graphs. The  paper is structured as follows. Section \ref{Basic} covers  basic results on the super domination number of a graph, including a characterization of graphs of order $n$ with $\gamma_{\rm sp}(G)=n-1.$ These graphs play an important role when studying the super domination number of  lexicographic product graphs. 
Section \ref{SectionLexicographicProduct} is devoted to the study of the super domination number of lexicographic product graphs. In particular, in Subsection \ref{SubsectionGeneralbounds} we obtain general bounds for the super domination number of  lexicographic product graphs in terms of some invariants of the factor graphs involved in the product. In Subsection \ref{SubsectionClosedFormulas} we 
show that the problem of finding the super domination number  of a graph is NP-Hard. We also study several families of graphs for which the bounds obtained previously are achieved. Finally, in Subsection \ref{SubsectionJoin} we obtain formulas for the  super domination number of join graphs.

For the remainder of the paper, definitions will be introduced whenever a concept is needed.

\section{Some remarks on the super domination number}
\label{Basic}

In this section we recall basic properties of the super domination number and give the full characterisation of graphs  of order $n$ with $\gamma_{\rm sp}(G)=n-1.$ To begin with, we introduce some 
notation and terminology. The {\it  closed
neighbourhood} of a vertex  $v$ is defined as $N[v]=N(v)\cup \{v\}$ and the degree of   $v $ is $d(v)=|N(v)|$. 
If $G$ has $n$ vertices and $d(v) =n-1$, then $v$ is a {\it universal} vertex of $G$.

\begin{theorem}{\rm \cite{MR3396565}} 
\label{theorem1} 
Let $G$ be a graph of order $n$. Then the following assertions hold.
\begin{itemize}
\item $\gamma_{\rm sp}(G)=1$ if and only if $G\cong K_1$  or $G\cong K_2$.
\item $\gamma_{\rm sp}(G)=n$ if and only if $G$ is an empty graph.
\item $\gamma_{\rm sp}(G)\ge  \lceil \frac{n}{2}\rceil$.
\end{itemize}
\end{theorem}

It is well known that for any graph $G$ without isolated vertices, $1\leq \gamma(G)\le  \lceil \frac{n}{2}\rceil $, so from the theorem above  we have that for any connected graph $G$,
\begin{equation}\label{TrivialBoundsontheSuperDominationNumber}
1\leq \gamma(G)\le  \left \lceil \frac{n}{2} \right\rceil \le \gamma_{\rm sp}(G)\leq n-1.
\end{equation}

Graphs with $\gamma_{\rm sp}(G)= n-1$ will play an important role in the study of the super domination number of lexicographic product graphs. 
In order to characterize these graphs we need to prove the following two lemmas. 

\begin{lemma}\label{lemma2} Let $G$ be a graph of order $n$. If $\gamma_{\rm sp}(G)=n-1,$ then $G$ is $P_4$-free and $C_4$-free.
\end{lemma}
\begin{proof}
Suppose that there exists $V'=\{x,y,w,z\}\subseteq V(G)$ such that the subgraph of $G$ induced by $V'$ is isomorphic to a path $P_4=(x,y,w,z)$ or a cycle  $C_4=(x,y,w,z,x)$. Then $V(G) \setminus \{x,z\}$ is a super dominating set of $G$, which implies  that $\gamma_{\rm sp}(G)\leq  n-2$.
\end{proof}
\begin{lemma}\label{lemma3} Let $G$ be a connected graph of order $n$. If there is no universal vertex in $G$, then $\gamma_{\rm sp}(G)\leq n-2.$
\end{lemma}
\begin{proof} Suppose that $\gamma_{\rm sp}(G)=n-1$ and $G$ does not have a universal vertex. Let $x$ be a vertex of   maximum degree in $G$. Since $d(x)<n-1$, there exists $z$ such that the distance between $x$ and $z$ is equal to two, and let denote by $y$ a common neighbour of $x$ and $z$. Suppose now that there exists $w\in N(x)$ such that $yw\not \in E(G)$. In such a case, the subgraph induced by the set $\{w,x,y,z\}$ is isomorphic to $P_4$ or $C_4$, which is a contradiction with Lemma~\ref{lemma2}. Hence, $N(x)\subseteq N[y]$. Furthermore, $z\in N(y)\setminus N(x)$, which implies $d(y)\geq d(x)+1$, which is a contradiction.
\end{proof}

To describe graphs with $\gamma_{\rm sp}(G)=n-1,$ we define a family $\mathcal{F}$ of graphs in the following way. 

\begin{itemize}
\item Let $k$ and $k'$ be two positive integers such that $k'=k$ or $k'=k-1$.
\item Let $\{G_i=(V_i,E_i):\; i=1,\dots, k\}$ be a family  of complete graphs.
\item Let $\{G_i'=(V_i',E_i'):\; i=1,\dots, k'\}$ be a family  of empty graphs.
\item Let $X_1=\bigcup_{i=1}^k E_i$ and  $X_2=\{xy\colon x\in V_i,y\in V_j' \textrm{ and }1\leq i\leq j\leq k'\}$.

\item For $k\ge 2$  we define  $X_3= \{xy\colon x\in V_i, y\in V_j \textrm{ and } 1\le i<j\le k\}$, while for $k=1$ we assume that $X_3=\emptyset$. 
\item With the notation above, we say that $G\in \mathcal{F}$ if    $V(G)=\left(\bigcup_{i=1}^k V_i \right)\cup  \left(\bigcup_{i=1}^{k'} V_i' \right)$ and $E(G)=X_1\cup X_2\cup X_3$ for some  integers $k$ and $k'$.  
\end{itemize}

\begin{figure}[h]
\centering
\begin{tikzpicture}[transform shape, inner sep = .7mm]
\draw(-1,-1) -- (1,1)--(1,-1)--(-1,1)--(-2,0)--(-1,-1)--(1,-1)--(1,1)--(-1,1)--(-1,-1);
\draw(1,-1)--(3.5,0)--(1,1);\draw(1,1)--(2,0)--(1,-1);\draw(-1,1)--(2,0)--(-1,-1);\draw(-1,1)--(3.5,0)--(-1,-1);
\node  at  (0,-1.5) {};

\node [draw=black, shape=circle, fill=white] at  (-1,-1) {};
\node [draw=black, shape=circle, fill=white] at  (-1,1) {};
\node [draw=black, shape=circle, fill=white] at  (1,-1) {};
\node [draw=black, shape=circle, fill=white] at  (1,1) {};
\node [draw=black, shape=circle, fill=white] at  (-2,0) {};
\node [draw=black, shape=circle, fill=white] at  (2,0) {};
\node [draw=black, shape=circle, fill=white] at  (3.5,0) {};
\node [right] at (3.6,0) {$^{x'}$};
\node [right] at (2.1,0) {$^{y'}$};
\node [left] at (-2,0) {$^{a'}$};
\node [above] at (-1,1) {$^{a}$};
\node [below] at (-1,-1.1) {$^{b}$};
\node [above] at (1,1) {$^{x}$};
\node [below] at (1,-1.1) {$^{y}$};
\end{tikzpicture} 

\caption{A graphs belonging to the family $\mathcal{F}$, where $k=k'=2$, $V_1=\{a,b\}$, $V'_1=\{a'\}$, $ V_2=\{x,y\}$ and  $V'_2=\{x',y'\} $.}\label{figExampleOfFamilyF}
\end{figure}
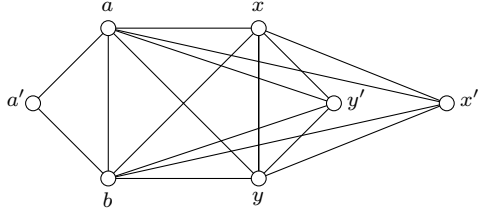

Figure \ref{figExampleOfFamilyF} shows an example of graph belonging to the family $\mathcal{F}$.
The following remark is a direct consequence of the definition of $\mathcal{F}$.

\begin{remark}\label{RemarkFamilyF}Let $G\in \mathcal{F}$. Then the following assertions hold for $x,y\in V(G)$.
\begin{itemize}
\item If $x\in V_i$, then $N[x]=\left(\bigcup_{j=1}^k V_j\right) \cup \left(\bigcup_{j=i}^{k'}V_j'\right)$.
\item If $y\in V_j'$, then $N(y)=\bigcup_{i=1}^j V_i$.
\item If $x,y\in V_j$, then $N[x]=N[y]$.

\item If  $x,y\in V_j'$, then $N(x)=N(y)$.

\item If $i<j$, $x\in V_i$ and $y\in V_j$, then $N[y]\subseteq N[x]$.
\item If $i<j$, $x\in V_i'$ and $y\in V_j'$, then $N(x)\subseteq N(y)$.

\item If $x\in V_i$ and $y\in V_j'$, then $N(y)\subseteq N(x)$
\end{itemize}
\end{remark}
 
\begin{theorem}
\label{familyF}
Let $G$ be a connected graph of order $n$.  Then $\gamma_{\rm sp}(G)=n-1$ if and only if $G\in \mathcal{F}.$
\end{theorem}

\begin{proof}
From Remark \ref{RemarkFamilyF} we deduce that if $G\in \mathcal{F}$, then $\gamma_{\rm sp}(G)=n-1$. 

From now on we assume that $\gamma_{\rm sp}(G)=n-1$. Thus, by Lemma~\ref{lemma3} we can claim that $G$ has at least one universal vertex. Let $V_1$ be the set of universal vertices of $G$ and  $V_1'=\{x\in V(G)\colon N(x)=V_1\}$. If $V(G)=V_1\cup V_1'$, then $G\in \mathcal{F}$. 
Now, if $V(G)\setminus (V_1\cup V_1')\ne \emptyset$, then 
  we denoted by $H_2$ the subgraph of $G$ induced by $V(G)\setminus (V_1\cup V_1')$.  If the subgraph $H_2$  has no universal vertex, then by Lemma~\ref{lemma3} there exists a super dominating set $D$ of $H_2$ such that $|D|\leq |V(H_2)|-2$, which is  a contradiction. So $H_2$ has at least one universal vertex.  
Let $V_2$ be the set of universal vertices of $H_2$ and  $V_2'=\{x\in V(H_2)\colon N(x)=V_2\}$. If
 $V(H_2)=V_2\cup V_2'$, then $H_2\in \mathcal{F}$, which also implies that $G\in \mathcal{F}$\footnote{Notice that if $V(H_2)=V_2\cup V_2'$, then $k=2$ and if $V_2'=\emptyset$, then $k'=1$, otherwise $k'=2$.}. 
Analogously, if  $V(H_2)\setminus (V_2\cup V_2')\ne \emptyset$, then we denote by $H_3$ the 
  subgraph of $H_2$ induced by $V(H_2)\setminus (V_2\cup V_2')$. 
 Next we repeat this process for $H_3$  to conclude that $H_3\in \mathcal{F}$ and, since $G$ is a finite graph, we continue the process until 
 $V(H_k)=V_k\cup V_k'$ for some $k$, where $V_k'$ may be empty, to conclude that $G\in \mathcal{F}$.
\end{proof}

\section{Super domination in lexicographic product of graphs}\label{SectionLexicographicProduct}

Let $G$ be a graph of order $n$ such that $V(G)=\{u_1,\ldots ,u_n\}$ and let $\mathcal{H}=\{H_1,H_2,\ldots,H_n\}$ be  an ordered family formed by $n$ graphs such that  $H_i$ corresponds to $u_i$ for every $i$.  The \emph{lexicographic product} of $G$ and $\mathcal{H}$ is the graph $G \circ \mathcal{H}$, such that $V(G \circ \mathcal{H})=\bigcup_{u_i \in V(G)} (\{u_i\} \times V(H_i))$ and $(u_i,v_r)(u_j,v_s) \in E(G \circ \mathcal{H})$ if and only if $u_iu_j \in E(G)$ or $i=j$ and $v_rv_s \in E(H_i)$. 
Figure \ref{figExampleOfLexiFamily} shows the lexicographic product of $P_3=(u_1,u_2,u_3)$ and the ordered family of graphs $\{P_4,K_2,P_3\}$, and the lexicographic product of $P_4=(u_1,u_2,u_3,u_4)$ and the family $\{H_1,H_2,H_3,H_4\}$, where $H_1 \cong H_4 \cong K_1$ and $H_2 \cong H_3 \cong K_2$. In general, we can construct the graph $G\circ\mathcal{H}$ by taking one copy of each $H_i\in\mathcal{H}$ and joining by an edge every vertex of $H_i$ with every vertex of $H_j$ for every $u_i u_j\in E(G)$.

\begin{figure}[!ht]
\centering
\begin{tikzpicture}[transform shape, inner sep = .7mm]
\pgfmathsetmacro{\espacio}{1};
\node [draw=black, shape=circle, fill=white] (v5) at (3,0) {};
\node [draw=black, shape=circle, fill=white] (v6) at (3,3*\espacio) {};
\draw[black] (v5) -- (v6);
\foreach \ind in {1,...,4}
{
\pgfmathsetmacro{\yc}{(\ind-1)*\espacio};
\node [draw=black, shape=circle, fill=white] (v\ind) at (0,\yc) {};
\draw[black] (v5) -- (v\ind);
\draw[black] (v6) -- (v\ind);
\ifthenelse{\ind>1}
{
\pgfmathtruncatemacro{\bind}{\ind-1};
\draw[black] (v\ind) -- (v\bind);
}
{};
}
\node [draw=black, shape=circle, fill=white] (v7) at (6,0) {};
\node [draw=black, shape=circle, fill=white] (v8) at (6,1.5*\espacio) {};
\node [draw=black, shape=circle, fill=white] (v9) at (6,3*\espacio) {};
\draw[black] (v7) -- (v8) -- (v9);
\foreach \ind in {7,...,9}
{
\draw[black] (v5) -- (v\ind);
\draw[black] (v6) -- (v\ind);
}
\end{tikzpicture}
\hspace{0.5cm}
\begin{tikzpicture}[transform shape, inner sep = .7mm]

\draw(-1,-1) -- (1,1)--(1,-1)--(-1,1)--(-2,0)--(-1,-1)--(1,-1)--(2,0)--(1,1)--(-1,1)--(-1,-1);
\node  at  (0,-1.5) {};

\node [draw=black, shape=circle, fill=white] at  (-1,-1) {};
\node [draw=black, shape=circle, fill=white] at  (-1,1) {};
\node [draw=black, shape=circle, fill=white] at  (1,-1) {};
\node [draw=black, shape=circle, fill=white] at  (1,1) {};
\node [draw=black, shape=circle, fill=white] at  (-2,0) {};
\node [draw=black, shape=circle, fill=white] at  (2,0) {};

\end{tikzpicture} 
\caption{The lexicographic product graphs $P_3\circ\{P_4,K_2,P_3\}$ and $P_4\circ\{H_1,H_2,H_3,H_4\}$, where $H_1 \cong H_4 \cong K_1$ and $H_2 \cong H_3 \cong K_2$.}\label{figExampleOfLexiFamily}
\end{figure}
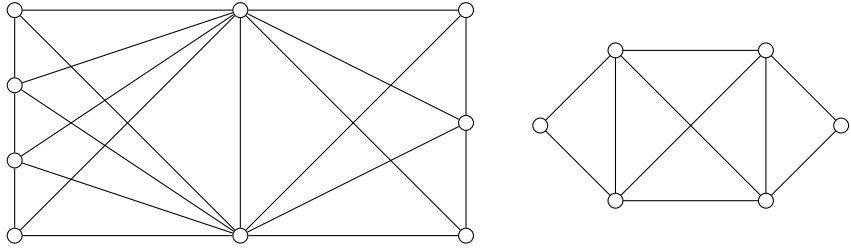

We will restrict our study to two particular cases. First, the traditional lexicographic product graph, where $H_i \cong H$ for every $i \in \{1,\ldots,n\}$, which is denoted as $G \circ H$ for simplicity \cite{Hammack2011,Imrich2000}.  
The other particular case we will focus on is the join of $G$ and $H$. The \emph{join graph} $G+H$\label{g join} is defined as the graph obtained from disjoint graphs $G$ and $H$ by taking one copy of $G$ and one copy of $H$ and joining by an edge each vertex of $G$ with each vertex of $H$ \cite{Harary1969,Zykov1949}. 
Note that $G+H\cong K_2\circ\{G,H\}$. The join operation is commutative and associative. 

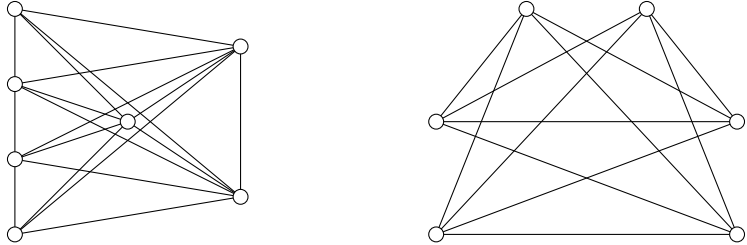
\begin{figure}[h]
\centering
\begin{tabular}{cccccc}

\begin{tikzpicture}[transform shape, inner sep = .7mm]

\draw(0,0) -- (0,3);
\draw(1.5,1.5) -- (3,0.5) -- (3,2.5) -- cycle;
\draw(1.5,1.5) -- (0,0);
\draw(3,0.5) -- (0,0);
\draw(3,2.5) -- (0,0);
\draw(1.5,1.5) -- (0,1);
\draw(3,0.5) -- (0,1);
\draw(3,2.5) -- (0,1);
\draw(1.5,1.5) -- (0,2);
\draw(3,0.5) -- (0,2);
\draw(3,2.5) -- (0,2);
\draw(1.5,1.5) -- (0,3);
\draw(3,0.5) -- (0,3);
\draw(3,2.5) -- (0,3);

\node [draw=black, shape=circle, fill=white] at  (0,0) {};
\node [draw=black, shape=circle, fill=white] at  (0,1) {};
\node [draw=black, shape=circle, fill=white] at  (0,2) {};
\node [draw=black, shape=circle, fill=white] at  (0,3) {};
\node [draw=black, shape=circle, fill=white] at  (1.5,1.5) {};
\node [draw=black, shape=circle, fill=white] at  (3,0.5) {};
\node [draw=black, shape=circle, fill=white] at  (3,2.5) {};

\end{tikzpicture} & & \hspace*{0.7cm} & &
\begin{tikzpicture}[transform shape, inner sep = .7mm]
\draw(0,0) -- (4,0);
\draw(0,1.5) -- (4,1.5);
\draw(0,0) -- (4,1.5);
\draw(0,1.5) -- (4,0);
\draw(1.2,3) -- (0,0);
\draw(1.2,3) -- (0,1.5);
\draw(1.2,3) -- (4,0);
\draw(1.2,3) -- (4,1.5);
\draw(2.8,3) -- (0,0);
\draw(2.8,3) -- (0,1.5);
\draw(2.8,3) -- (4,0);
\draw(2.8,3) -- (4,1.5);

\node [draw=black, shape=circle, fill=white] at  (0,0) {};
\node [draw=black, shape=circle, fill=white] at  (0,1.5) {};
\node [draw=black, shape=circle, fill=white] at  (4,0) {};
\node [draw=black, shape=circle, fill=white] at  (4,1.5) {};
\node [draw=black, shape=circle, fill=white] at  (1.2,3) {};
\node [draw=black, shape=circle, fill=white] at  (2.8,3) {};

\end{tikzpicture} \\
 
\end{tabular}
\caption{Two join graphs:  $P_4+C_3 \cong K_2\circ \{P_4,C_3\}$ and $N_2+N_2+N_2 \cong K_3\circ N_2$.}
\label{ex join}
\end{figure}

Moreover, complete $k$-partite graphs,  $$K_{p_1,p_2,\dots,p_k}\cong K_n \circ \{N_{p_1},N_{p_2},\dots ,N_{p_k}\}\cong N_{p_1}+N_{p_2}+\cdots +N_{p_k},$$ are typical examples of join graphs, where $N_{p_i}$ denotes the empty graph of order $p_i$. The particular case   illustrated in Figure \ref{ex join} (right hand side), is no other than the complete $3$-partite graph $K_{2,2,2}$.
   
Notice that  for any $g\in G$ and any graph $H$,  the subgraph of $G\circ H$ induced by $\{g\}\times V(H)$ is isomorphic to $H$.

\begin{remark} Let $G$ and $H$ be two graphs. Then the following assertions hold.

\begin{itemize}
\item $G\circ H$ is connected if and only if   $G$ is connected.
\item  If $G=G_1\cup \ldots \cup G_t,$ then $G\circ H= (G_1\circ H) \cup \ldots \cup (G_t \circ H).$ 
\end{itemize}
\end{remark}

According to the remark above,  we can restrict ourselves to the case of lexicographic product graphs $G\circ H$ for which $G$ is connected. For basic properties of the lexicographic product of two graphs we suggest the handbook by
Hammack, Imrich and Klav{\v{z}}ar
\cite{Hammack2011}.

A main problem in the study of product of graphs consists of finding exact values or sharp
bounds for specific parameters of the product of two graphs and express
these in terms of invariants of the factor graphs. In particular,   we cite the following works on domination theory of lexicographic product graphs. For instance, the domination number was studied in \cite{MR3363260,Nowakowski1996}, the Roman domination number was studied in \cite{SUmenjak:2012:RDL:2263360.2264103}, the rainbow domination number was studied in \cite{MR3057019}, while the doubly connected domination number was studied in \cite{MR3200151}. 

 To begin our study we need to introduce the following additional notation.
Given   $g\in V(G) $ and $W\subseteq V(G)\times V(H)$ we define $$W_g=\{h:\, (g,h)\in W\}.$$ 
For simplicity, the neighbourhood of $(x,y)\in V(G)\times V(H)$ will be denoted by $N(x,y)$ instead of $N((x,y))$.

\begin{lemma}\label{important} Let $G$ be a graph and let $H$ be a nonempty graph. If $W$ is a $\gamma_{\rm sp}(G \circ H)$-set, then $|W_g|\ge \gamma_{\rm sp}(H)$ for every  $g\in V(G)$.
\end{lemma}
\begin{proof}Since $H$ is a nonempty graph, if $|W_g|\ge |V(H)|-1$, then we are done. Assume that $|W_g|\le |V(H)|-2$ and let $h,h'\in \overline{W}_g$.  Since 
$$N(g,h)\cap \left [(V(G)\setminus  \{g\})\times V(H) \right)]=N(g,h')\cap \left [(V(G)\setminus  \{g\})\times V(H) \right)],$$
 we can conclude that $(g,h)$ (and also $(g,h')$) has a  private neighbour with respect to  $\overline{W}$ which belongs to $\left(\{g\}\times V(H)\right)\cap W$. Hence, $h$ (and also $h'$) has a   private neighbour with respect to  $\overline{W}_g$ which belongs to $W_g$. Therefore, $W_g$ is a super dominating set for $H$, which implies that $|W_g|\ge \gamma_{\rm sp}(H)$.  
\end{proof}

\begin{lemma}
\label{adjacentcopies}
Let $G$ and $H$ be two graphs. Let $xx'\in E(G)$ and   $W$   a $\gamma_{\rm sp}(G\circ H)$-set. If $\overline{W}_x\not=\emptyset$ and $\overline{W}_{x'}\not=\emptyset$, then $|\overline{W}_x|=|\overline{W}_{x'}|=1.$
\end{lemma}
\begin{proof}
Suppose that $(x,y),(x',y'),(x',y'')\in \overline{W}$. Since $xx'\in E(G)$, we have that $N(x',y'')\subseteq N(x',y')\cup N(x,y)$, which is a contradiction. Hence, $|\overline{W}_{x'}|\le 1$. Therefore, the results follows.
\end{proof}

\subsection{General bounds}\label{SubsectionGeneralbounds}

Recall that an \textit{independent  set} of a graph $G$ is a subset $S\subseteq V(G)$  such that no two vertices in $S$ represent an edge of $G$, \textit{i.e.},  $N(x)\cap S=\emptyset$, for every $x\in S$.  The  cardinality of a maximum independent   set of $G$ is called the \emph{independence number} of $G$ and is denoted by $\alpha(G)$.
We refer to an $\alpha(G)$-set in a graph $G$ as an independent set of cardinality $\alpha(G)$.

  Fink, Jacobson \cite{MR812672} defined \textit{k-independent set} of a graph $G$ as a set $S\subseteq V(G)$ such that the subgraph induced by $S$ has maximum degree at most $k-1$,  \textit{i.e.},  $|N(x)\cap S|\le k-1$, for every $x\in S$.  The  cardinality of a maximum $k$-independent  set of $G$ is called the \emph{k-independence number} of $G$ and is denoted by $\alpha_k(G)$. Obviously any $1$-independent set of $G$ is an independent set of $G$.

\begin{theorem}\label{MainTheorem}  For any   nonempty graph  $H$ of order $n'$ and for any graph $G$ of order~$n$, $$\gamma_{\rm sp}(G \circ H) \le \alpha(G) \gamma_{\rm sp}(H)+(n-\alpha(G))n'.$$
In particular,  if $\gamma_{\rm sp}(H)= n'-1$ and $H\not \cong K_{n'}$,   then $$\gamma_{\rm sp}(G \circ H) \le  nn'-\alpha_2(G).$$
   \end{theorem}

\begin{proof}
Let $S_1 $ be an $\alpha(G)$-set and $S_2$ a $\gamma_{\rm sp}(H)$-set.  We claim that 
\begin{equation}\label{Definitionof-S}
S= \left (\bigcup_{x\in S_1}\{x\}\times S_2\right)  \cup \left(\bigcup_{x\notin S_1}\{x\}\times V(H)\right) 
\end{equation}
 is a super dominating set of $G\circ H.$ To see this we set $(x,y)\notin S.$ Hence, $x\in S_1$ and $y\in \overline{S_2}$, so that there exists $y'\in S_2$ such that $N(y')\cap \overline{S_2}=\{y\}$, which implies that  $(x,y')\in S$   and $N(x,y')\cap \overline{S}=\{(x,y)\}$. Thus, $S$ is a super dominating set of $G\circ H$ and, as a consequence,
$$\gamma_{\rm sp}(G\circ H)\leq |S|=|S_1||S_2|+(n-|S_1|)n'=\alpha(G) \gamma_{\rm sp}(H)+(n-\alpha(G))n'.$$ 

Now, let $\gamma_{\rm sp}(H)= n'-1$ and $H\not \cong K_{n'}$. In this case we take the sets $S_1$ and $S_2$ in a different manner, \textit{i.e.}, we take $S_1 $ as an $\alpha_2(G)$-set and $S_2=V(H)\setminus \{y\}$, where $y$ is a nonuniversal vertex of $H$.  We claim that the set $S$ defined by \eqref{Definitionof-S} is a super dominating set of $G\circ H.$ To see this we take $(x,y)\notin S.$ Since $x\in S_1$, we have two possibilities, namely, (a) $N(x)\cap S_1=\emptyset$ or (b)
there exists $x'\in V(G)\setminus \{x\}$ such that $N(x)\cap S_1=\{x'\}$. In case (a), for every $y'\in S_2$ we have that  $(x,y')\in S$  and $N(x,y')\cap \overline{S}=\{(x,y)\}$, while in case  (b),  for every $y'\in V(H)\setminus N[y]$ we have that  $(x',y')\in S$ and  $N(x',y')\cap \overline{S}=\{(x,y)\}$. 
Thus, $S$ is a super dominating set of $G\circ H$ and, as a consequence,
$$\gamma_{\rm sp}(G\circ H)\leq |S|=|S_1|(n'-1)+(n-|S_1|)n'=nn'-\alpha_2(G).$$ 
Therefore, the result follows.
\end{proof}

As we will show in Theorem \ref{TheoremEquality} and Propositions \ref{G-Complete}, \ref{Bipartite}, \ref{PropositionCycles} and \ref{PropositionPaths}, the bounds above are achieved by several families of graphs. 

Notice that the bound $\gamma_{\rm sp}(G \circ H) \le \alpha(G) \gamma_{\rm sp}(H)+(n-\alpha(G))n'$ is never better than 
$\gamma_{\rm sp}(G \circ H) \le  nn'-\alpha_2(G)$, as $\alpha(G)\le \alpha_2(G)$.

A {\it vertex cover} of  $G$ is a set $X\subseteq V(G)$  such that each 
edge of $G$ is incident to at least one vertex of $X$.   The {\it vertex cover 
number} $\tau(G)$ is the cardinality of a minimum vertex cover of $G$. A vertex 
cover of cardinality $\tau(G)$ is called a $\tau(G)$-set. 
The following well-known result, due to Gallai \cite{Gallai1959}, states the relationship between the independence number and the vertex cover number of a graph.
 
\begin{theorem}{\em\cite{Gallai1959}}{\rm (Gallai, 1959)}
\label{th gallai}
For any graph $G$ of order $n$, $\alpha(G)+\tau(G) = n.$
\end{theorem}

By Theorems \ref{MainTheorem} and \ref{th gallai} we deduce that $$\gamma_{\rm sp}(G \circ H) \le  \alpha(G) \gamma_{\rm sp}(H)+\tau(G)n'.$$

\begin{theorem}\label{TrivialLowerBound} Let $G$ be a  graph of order $n\ge 2$ and let $H$ be a nonempty graph of order $n'$. Then the following assertions hold.
\begin{itemize}
\item $\gamma_{\rm sp}(G \circ H)\geq n\gamma_{\rm sp}(H).$
\item $\gamma_{\rm sp}(G \circ H)= n\gamma_{\rm sp}(H)$ if and only if $G\cong K_2$,  $\gamma_{\rm sp}(H)=n'-1$ and $H\not\cong K_{n'}$.
\end{itemize}
\end{theorem}

\begin{proof}
The lower bound is a direct consequence of Lemma \ref{important}. Assume that $G\cong K_2$,  $\gamma_{\rm sp}(H)=n'-1$ and $H\not\cong K_{n'}$. By the lower bound we have 
$\gamma_{\rm sp}(G \circ H)\geq 2(n'-1).$ Let $h$ be a nonuniversal vertex of $H$ and   $V(G)=\{a,b\}$. To show that $\gamma_{\rm sp}(G \circ H)\leq 2(n'-1)$ we only need to observe that $V(G)\times V(H)\setminus \{(a,h),(b,h)\}$ is a super dominating set for $G\circ H$, \textit{i.e.},  if $h'$ is not adjacent to $h$ in $H$, then $(a,h')\in N(b,h) \setminus N(a,h)$ and $(b,h')\in N(a,h)\setminus N(b,h)$.

From now on we assume that $\gamma_{\rm sp}(G \circ H)= n\gamma_{\rm sp}(H)$.  Let $W$ be a $\gamma_{\rm sp}(G \circ H)$-set. Since $\gamma_{\rm sp}(G \circ H)= n\gamma_{\rm sp}(H)$, from Lemma \ref{important} we deduce that for any $g\in V(G)$, $|W_g|=\gamma_{\rm sp}(H)\le n'-1$. Suppose that $G\not \cong K_2$. Let $x\in V(G)$ be a vertex of degree at least two and let  $(u,v)\in W$ be a private neighbour of $(x,y)\not \in W$ with respect to  $\overline{W}$, \textit{i.e.} $N(u,v)\cap \overline{W}=\{(x,y)\}$. Let $x'\in V(G)\setminus \{u\}$ be a neighbour of $x$ and $(x',y')\not \in W$. If $u=x$, then $(x',y')\in N(u,v)$, and  $\{(x,y),(x',y')\}\subseteq N(u,v)\cap \overline{W}$, which is a contradiction, so that $u\ne x$.  Thus, if $u$ has degree one, then for $(u,z)\not\in W$   we have $N(u,z)\subset N(x,y)\cup N(x',y')$ , which is a contradiction. Otherwise there exists $u'\in N(u)\setminus \{x\}$ and for  
$(u',z')\not\in W$   we have $(u',z')\in N(u,v)$. Thus, $\{(x,y),(u',z')\}\subseteq N(u,v)\cap \overline{W}$, which is a contradiction again. Hence, we can conclude that $G\cong K_2$. 

Notice that $H\not \cong K_{n'}$, as $K_2\circ K_{n'}\cong K_{2n'}$ and $\gamma_{\rm sp}(K_{2n'})=2n'-1>2(n'-1)=2\gamma_{\rm sp}(K_{n'})$.

To conclude the proof suppose that $\gamma_{\rm sp}(H)\le n'-2$. In this case, since $\gamma_{\rm sp}(G \circ H)= n\gamma_{\rm sp}(H)$, from Lemma \ref{important} we deduce that for any $g\in V(G)$, $|W_g|=\gamma_{\rm sp}(H)\le n'-2$. Let $V(K_2)=\{a,a'\}$. Hence, for any $(a,b_1),(a,b_2),(a',b_3)\not \in W$ we have that $N(a,b_1)\subseteq N(a,b_2) \cup N(a',b_3)$, which is a contradiction. Therefore,  $\gamma_{\rm sp}(H)=n'-1$ and so the result follows.
\end{proof}

The following result provides an upper bound on the super domination number of the graph $G \circ H$ in terms of the order and the super domination number of its factors.
 
\begin{theorem}\label{BoundMin}
For any graph $G$ of order $n\ge 2$ and any  graph $H$ of order $n'\ge 2$,
$$\gamma_{\rm sp}(G \circ H)\leq \min\{ n(n'-1)+\gamma_{\rm sp}(G), n'(n-1)+\gamma_{\rm sp}(H)\}.$$
\end{theorem}
\begin{proof}Let $S$ be a $\gamma_{\rm sp}(G)$-set and let $y\in V(H)$. We claim that 
$$W=\left( \bigcup_{x\in S}\{x\}\times V(H)\right)\cup \left(\bigcup_{x\not \in S}\{x\}\times (V(H)\setminus \{y\}) \right) $$
 is a super dominating set for $G\circ H$. To see this we only need to observe that for any $(x,y) \in \overline{W}$, there exists $x'\in S$ such that $N(x')\cap \overline{S}=\{x\}$, which implies that
 $N(x',y')\cap \overline{W}=\{(x,y)\}$  for every $y'\in V(H)$.
 Hence, $\gamma_{\rm sp}(G \circ H)\leq |W|= n(n'-1)+\gamma_{\rm sp}(G)$.
 
 Now, let $S'$ be a $\gamma_{\rm sp}(H)$-set and let $x'\in V(G)$. We claim that 
$$W'=\left( \bigcup_{x\ne x'}\{x\}\times V(H)\right)\cup \left(\{x'\}\times S'   \right) $$
 is a super dominating set for $G\circ H$. In this case we only need to observe that $\{x'\}\times S'$ is a super dominating set for  
the subgraph induced by  $\{x'\}\times V(H)$.
 Hence, $\gamma_{\rm sp}(G \circ H)\leq |W'|= n'(n-1)+\gamma_{\rm sp}(H)$.
\end{proof}

The bound above is tight.
For instance, as we will show in Proposition \ref{G-Complete}, the equality $\gamma_{\rm sp}(G \circ H)=n'(n-1)+\gamma_{\rm sp}(H)$ holds  for any graph $H$ with $\gamma_{\rm sp}(H)\le n'-2$ and $G$ isomorphic to a complete graph. On the other hand, it is not difficult to check that if $G\cong K_r \odot K_1$ is the corona product  $K_r$ times $K_1$, then $\gamma_{\rm sp}(G)=\frac{n}{2}=r$ and $\gamma_{\rm sp}(G \circ K_{n'})=2rn'-r=n(n'-1)+\gamma_{\rm sp}(G)$.

In the next result we obtain an upper bound for $\gamma_{\rm sp}(G \circ N_{n'})$, where $N_{n'}$ is the empty graph of order $n'$. 
To state the result, we first need some additional notation and terminology. 
A set $S$ of vertices is called a  2-{\it packing} if for every pair of vertices $u,v\in S$, $N [u]\cap N [v]=\emptyset$. The 2-packing number  $\rho(G)$ of a graph $G$ is the  cardinality of a maximum 2-packing in $G$. A 2-packing  of cardinality $\rho(G)$ is called a $\rho(G)$-set.
Given a graph $G$, its {\it line graph} $L(G)$ is a graph such that  each vertex of $L(G)$ represents an edge of $G$ and
    two vertices of $L(G)$ are adjacent if and only if their corresponding edges are adjacent in $G$.

\begin{theorem}\label{LowerBoundEmptyGraphs}
For any  graph $G$ of order $n$ and any integer $n'\ge 2$,
$$\gamma_{\rm sp}(G\circ N_{n'})\le nn'-2\rho(L(G)).$$
\end{theorem}

\begin{proof}
Let $S$ be a $\rho(L(G))$-set and $V_S=\{x\in V(G):\, x\in e\text{ for some }e\in S\}$. Given a fix vertex $y'\in V(H)$ we define 
$$W=\left(\bigcup_{x\in V_S}\{x\}\times (V(N_{n'})\setminus \{y'\}) \right) \cup \left(\bigcup_{x\not \in V_S}\{x\}\times V(N_{n'})  \right).$$ Notice that if $(x,y')\not\in W$, then $x\in V_S$, which implies that there exists $x'\in N(x)$ such that $\{x,x'\}\in S$. Now, since $S$ is  a 2-packing  of $L(G)$, for any $y\in V(N_{n'})\setminus \{y'\}$ we have that $(x',y)$ is a private neighbour of $(x,y')$ with respect to $\overline{W}$,  \textit{i.e.}$N(x',y)\cap \overline{W}=\{(x,y')\}$. Hence, $W$ is a super dominating set of $G\circ H$ and so $\gamma_{\rm sp}(G\circ N_{n'})\le |W|=nn'-|V_S|=nn'-2\rho(L(G))$.
Therefore, the result follows.
\end{proof}

The bound above is tight. It is achieved, for instance, for 
$G\cong K_n$. Notice that $\rho(L(K_n))=1$, $K_n\circ N_{n'}\cong K_{n',\dots,n'}$ and $\gamma_{\rm sp}(K_n\circ N_{n'})=nn'-2=nn'-2\rho(L(K_n))$.


\subsection{Closed formulas and complexity}\label{SubsectionClosedFormulas}

In this subsection we obtain closed formulas for the super domination number of lexicographic product graphs and, as a consequence of the study, we show that the problem of computing the super domination number of a graph is NP-Hard.

\begin{theorem}\label{TheoremEquality} Let $G$   be a graph  of order~$n$  and maximum degree $\Delta(G)$. If  $H$  is a graph of order $n'$ such that $n'-\gamma_{\rm sp}(H)>\Delta(G)+1$, then
 $$\gamma_{\rm sp}(G \circ H) = \alpha(G) \gamma_{\rm sp}(H)+(n-\alpha(G))n'.$$
    \end{theorem}
    
   \begin{proof}
    By Theorem \ref{MainTheorem} we have that $\gamma_{\rm sp}(G \circ H) \le \alpha(G) \gamma_{\rm sp}(H)+(n-\alpha(G))n'$, so that it remains to show that $\gamma_{\rm sp}(G \circ H) \ge \alpha(G) \gamma_{\rm sp}(H)+(n-\alpha(G))n'.$
    
    Let $W$ be a $\gamma_{\rm sp}(G \circ H) $-set and set 
    $$X=\{x\in V(G):\, |W_x|<n'\}.$$
    We claim that $X$ is an independent set. To see this, suppose that  there are two adjacent vertices $x,x'$ which belong to $X$.  Notice that Lemma \ref{adjacentcopies}  leads to $|W_x|=|W_{x'}|=n'-1$. Hence, for any $\gamma_{\rm sp}(H)$-set $S$ we have that
    $$W'=\left( W\setminus \left( \{x\}\times V(H)\right)\right)\cup \left(\{x\}\times S\right) \cup \left(\bigcup_{u\in N(x)}\{u\}\times V(H) \right) $$
    is a super dominating set for $G\circ H$ and 
  \begin{align*}
  |W'|&=|W|+|X\cap N(x)|-(n'-1-\gamma_{\rm sp}(H))\\
  &\le |W|+|\Delta(G)|-(n'-1-\gamma_{\rm sp}(H))\\
  &<|W|,
  \end{align*}  
    which is a contradiction. Thus,
    $X$ is an independent set and, by Lemma \ref{important} we have that
\begin{align*}
   \gamma_{\rm sp}(G \circ H)&=\sum_{u \in X}|W_u|+\sum_{u\not \in X}|W_u|\\
   &\ge |X|\gamma_{\rm sp}(H)+(n-|X|)n'\\
   &=nn'-|X|(n'-\gamma_{\rm sp}(H))\\
      &\ge nn'-\alpha(G)(n'-\gamma_{\rm sp}(H))\\
   &= \alpha(G)\gamma_{\rm sp}(H)+(n-\alpha(G))n',
\end{align*}   
as required.
   \end{proof}

 Fernau and Rodr\'{i}guez-Vel\'{a}zquez \cite{RV-F-2013,MR3218546}  showed that the study of corona product graphs enables us to infer NP-hardness results for computing the (local) metric dimension, based on according NP-hardness results for the (local) adjacency dimension. Our next result shows how the  study of lexicographic product graphs enables us to infer an NP-hardness result for computing the super domination number, based on a well known NP-hardness result for the independence number, \textit{i.e.}, since the problem of computing the independence number of a graph is {\rm NP}-Hard \cite{Garey1979},   Theorem \ref{TheoremEquality} leads to the following result.

   \begin{corollary}
   The problem of finding the super domination number of a graph is {\rm NP}-Hard.
   \end{corollary}
   
   \begin{proof}
   Let $G$ be a graph of order $n$ and maximum degree $\Delta(G)$. By Theorem \ref{TheoremEquality}, for any integer $t> \Delta(G)+1$ we have that
   $$\gamma_{\rm sp}\left(G\circ \bigcup_{i=1}^tK_2\right)=t(2n-\alpha(G)).$$
   Therefore, since the problem of computing the independence number of a graph is {\rm NP}-Hard, we conclude that  the problem of finding the super domination number of a graph is {\rm NP}-Hard too.
   \end{proof}

The remaining results of this subsection concern the case in which we fix the first factor in the lexicographic product.  

\begin{proposition}\label{G-Complete}
Let $H$ be a noncomplete graph of order $n'$ and let $n\ge 2$ be an integer. Then the following assertions hold.
\begin{itemize}
\item If $\gamma_{\rm sp}(H)=n'-1$, then $\gamma_{\rm sp}(K_n \circ H)= nn'-2$.
\item If $\gamma_{\rm sp}(H)\le n'-2$, then $\gamma_{\rm sp}(K_n \circ H)= n'(n-1)+\gamma_{\rm sp}(H)$.
\end{itemize}
\end{proposition}

\begin{proof}
Since $\alpha(K_n)=1$ and $\alpha_2(K_n)=2$, by Theorem~\ref{MainTheorem} we immediately have that if $\gamma_{\rm sp}(H)=n'-1$, then $\gamma_{\rm sp}(K_n \circ H)\le  nn'-2$ and if $\gamma_{\rm sp}(H)\le n'-2$, then $\gamma_{\rm sp}(K_n \circ H)\le  n'(n-1)+\gamma_{\rm sp}(H)$.

Now, let $W$ be a $\gamma_{\rm sp}(K_n\circ H)$-set.  Notice that, since $K_n\circ H\not \in \mathcal{F}$, Theorem \ref{familyF} leads to $|\overline{W}|\ge 2$. If $(x,y),(x',y') \in \overline{W} $, for $x\ne x'$, then $N(u,v)\subseteq N(x,y)\cup N(x',y')$, for every $(u,v) \in V(K_n)\times V(H)\setminus \{(x,y),(x',y')\}$, which means that $|W|= nn'-2$.  Furthermore, if $(x,y),(x,y')\not \in W$, then $\overline{W}\subseteq \{x\}\times V(H)$, and by Lemma \ref{important} we have that  $|W|\ge (n-1)n'+ \gamma_{\rm sp}(H)$. Hence, $$\gamma_{\rm sp}(K_n \circ H)\ge  \min\{nn'-2, n'(n-1)+\gamma_{\rm sp}(H)\}.$$
Thus, if $\gamma_{\rm sp}(H)\le n'-2$, then $\gamma_{\rm sp}(K_n \circ H)\ge   n'(n-1)+\gamma_{\rm sp}(H)$. Finally, if  
 $\gamma_{\rm sp}(H)=n'-1$, then $\gamma_{\rm sp}(K_n \circ H)\ge nn'-2$.     
\end{proof}

For any complete bipartite graph $K_{r,t}$, where  $r\le t$ and $t\ge 2$, we have $ \alpha(K_{r,t})=\alpha_2(K_{r,t})=t$. Therefore, we can state the following proposition  which shows again that the bounds in Theorem \ref{MainTheorem}  are tight.

\begin{proposition}\label{Bipartite}
For any nonempty graph $H$ and any integers $r,t$, where $r\le t$ and $t\ge 2$,
  $$\gamma_{\rm sp}(K_{r,t} \circ H) = t \gamma_{\rm sp}(H)+rn'.$$ 
\end{proposition}

\begin{proof}
Since  $r\le t$ and $t\ge 2$, we have $ \alpha(K_{r,t})=\alpha_2(K_{r,t})=t$. Therefore,  Theorem \ref{MainTheorem} leads to $\gamma_{\rm sp}(K_{r,t} \circ H) \le  t \gamma_{\rm sp}(H)+rn'$.  

It remains to show that $\gamma_{\rm sp}(K_{r,t} \circ H) \ge  t \gamma_{\rm sp}(H)+rn'$. Let $V(K_{r,t})=V_r\cup V_t$ where the vertices in $V_r$ have degree $t$ and the vertices in $V_t$ have degree $r$. Let $W$ be a $\gamma_{\rm sp}(K_{r,t} \circ H)$-set.  By Theorem \ref{familyF} we have that $|\overline{W}|\ge 2$ and so   we can fix $(x,y),(a,b)\in \overline{W}$ and differentiate the following two cases.
 
\noindent Case 1.  $x\in V_r$ and $a\in V_t$.  By Lemma 
\ref{adjacentcopies}, 
$| \overline{W}_x|=|  \overline{W}_a|=1$ and for any $u\in V(K_{r,t})\setminus \{x,a\}$ we have that $\overline{W}_u=\emptyset$, so that  
\begin{equation}\label{ForContradiction}
t \gamma_{\rm sp}(H)+rn'\ge |W|\ge (r+t)n'-2.
\end{equation}
If  $t=2$ and $\gamma_{\rm sp}(H)=n'-1$, then  \eqref{ForContradiction} leads to $|W|=(r+2)n'-2=2(n'-1)+rn'$, as required.  Also, if $\gamma_{\rm sp}(H)\le n'-2$, then 
\eqref{ForContradiction} leads to $t\le 1$, which is a contradiction.
Now, if $t\ge 3$, then \eqref{ForContradiction} leads to $\gamma_{\rm sp}(H)\ge \left\lceil n'-\frac{2}{t}\right\rceil=n'$, which is a contradiction again.

\noindent Case 2. 
$x,a\in V_r$ or $x,a\in V_t$. If $t=2$, then $K_{r,t} \cong C_4$ or $K_{r,t} \cong P_3$, and we are done. 
If $t\ge 3$, then by Lemma \ref{important} we have that
$|W|=\sum_{u\in V_t}|W_u|+\sum_{u\in V_r}|W_u|\ge t\gamma_{\rm sp}(H) +rn',$
as required. 
\end{proof}

It is well known that for any integer $n\ge 3$, $\alpha(C_n)=\left \lfloor  \frac{n}{2}\right\rfloor$ and it is not difficult to check that $\alpha_2(C_n)=\left \lfloor  \frac{2n}{3}\right\rfloor$. In order to study the super domination number of $C_n \circ H$ we need to state the following lemma.

\begin{lemma}
\label{LemmaCycle}
Let  $n\ge 5$ be an integer and $V(C_n)=\{v_1,v_2,\dots , v_n\}$, where $v_i$ is adjacent to $v_{i+1}$ and the subscripts are taken modulo $n$.   If $S\subseteq V(C_n)$ and  $|S|=\left \lfloor  \frac{2n}{3}\right\rfloor +1$, then there exists a subscript $i$ such that   $\{v_i,v_{i+1},\ldots ,v_{i+4}\}\subseteq S$   or  $\{v_i,v_{i+2},v_{i+3},v_{i+4}\}\subseteq S$  or $\{v_i,v_{i+1},v_{i+2},v_{i+4}\}\subseteq S$.  
\end{lemma}
\begin{proof}   Let  $S\subseteq V(C_n)$ such that $|S|=\left \lfloor  \frac{2n}{3}\right\rfloor +1$. Suppose that $C_n$ does not contain a path $P=(v_i,v_{i+1},\ldots ,v_{i+4})$ such that all vertices of $P$ belong to $S$ or all but $v_{i+1}$ (or $v_{i+3}$) belong to $S$. 
If no path in this form contains three consecutive vertices of $C_n$, then $S$ is a $2$-independent set, which is a contradiction as $|S|=\left \lfloor  \frac{2n}{3}\right\rfloor +1>\alpha_2(C_n)$. Otherwise, the vertices in $X=S\cap V(P)$ are consecutive and   $|X|\le 4$.  Now, if $X=\{x_i,x_{i+1},x_{i+2}\}$, then $\{x_{i-1},x_{i-2},x_{i+3},x_{i+4}\}\cap S=\emptyset$ and  if $X=\{x_i,x_{i+1},x_{i+2},x_{i+3}\}$, then $\{x_{i-1},x_{i-2},x_{i+4},x_{i+5}\}\cap S=\emptyset$. Let $U$ be the set of these maximal paths where  $|X|=3$ and, analogously, let  $U'$ be the set these paths where  $|X|=4$. We assume that the labelling in all these paths is induced by the labelling in $C_n$. Next,  we can construct a set $S^*$ from $S$ by removing $x_{i+2}$ and adding $x_{i+3}$ for each path in  $U$, and by removing $x_{i+2}$ and adding $x_{i+4}$ for each path in  $U'$. Hence,  $S^*$  is a $2$-independent set, which is a contradiction, as $|S^*|=|S|=\left \lfloor  \frac{2n}{3}\right\rfloor +1>\alpha_2(C_n)$.
\end{proof}

\begin{proposition}\label{PropositionCycles} 
Let $H$ be a nonempty graph of order $n'$ and  let $n\ge 4$ be an integer. If $\gamma_{\rm sp}(H)= n'-1$ and $H\not \cong K_{n'}$,   then $$\gamma_{\rm sp}(C_n \circ H) =  nn'-\left \lfloor  \frac{2n}{3}\right\rfloor.$$
Furthermore,  if $\gamma_{\rm sp}(H)\le  n'-2$, then 
 $$\gamma_{\rm sp}(C_n \circ H) = \left \lfloor  \frac{n}{2}\right\rfloor \gamma_{\rm sp}(H)+n'\left \lceil  \frac{n}{2}\right\rceil.$$
\end{proposition}
\begin{proof}
Let $V(C_n)=\{x_1,x_2,\dots , x_n\}$, where $x_i$ is adjacent to $x_{i+1}$ and the subscripts are taken modulo $n$.  
Let $W$ be a $\gamma_{\rm sp}(C_n\circ H)$-set and 
$$X=\{x\in V(C_n):\, |W_x|<n'\}.$$

We first consider the case  $\gamma_{\rm sp}(H)= n'-1$ and $H\not \cong K_{n'}$. By Theorem \ref{MainTheorem} we have that $\gamma_{\rm sp}(C_n \circ H) \le  nn'-\left \lfloor  \frac{2n}{3}\right\rfloor.$ Suppose $\gamma_{\rm sp}(C_n \circ H) < nn'-\left \lfloor  \frac{2n}{3}\right\rfloor$.   From Lemma~\ref{important} we know that $|\overline{W_x}|= 1$ for every $x\in X$, so that $|\overline{W}|=|X|\geq \left \lfloor  \frac{2n}{3}\right\rfloor+1$. If $n=4$, then at least three  vertices, say $(x_1,y_1),(x_2,y_2),(x_3,y_3)$ belong to $\overline{W}$ and, in such a case,  $N(x_1,y_1)\subseteq N(x_2,y_2)\cup N(x_3,y_3)$, so $W$ is not a super dominating set of $C_4\circ H$, which is a contradiction. If  $n\geq 5$, then  by Lemma \ref{LemmaCycle}  there exists a subscript $i$ such that   $\{x_{i},x_{i+1},\ldots , x_{i+4}\}\subseteq X$ or $\{x_{i},x_{i+2},x_{i+3} , x_{i+4}\}\subseteq X$  or $\{x_{i},x_{i+1},x_{i+2} , x_{i+4}\}\subseteq X$.  In all these cases  $N(x_{i+2})\subseteq \cup_{j=i}^{i+4}N(x_j)$, so that  $W$ is not a super dominating set of $C_n\circ H$, which is a contradiction.  Therefore, $\gamma_{\rm sp}(C_n \circ H) =  nn'-\left \lfloor  \frac{2n}{3}\right\rfloor.$ 

From now on we assume that $\gamma_{\rm sp}(H)\le  n'-2$. By Theorem \ref{MainTheorem} we have that $\gamma_{\rm sp}(C_n \circ H) \le  \left \lfloor  \frac{n}{2}\right\rfloor \gamma_{\rm sp}(H)+n'\left \lceil  \frac{n}{2}\right\rceil.$
We will show that $\gamma_{\rm sp}(C_n \circ H) \ge  \left \lfloor  \frac{n}{2}\right\rfloor \gamma_{\rm sp}(H)+n'\left \lceil  \frac{n}{2}\right\rceil.$ 
If $X$ is an independent set, then by Lemma \ref{important},
\begin{align*}
|W|=& \sum_{x\in X}|W_x|+ \sum_{x\not\in X}|W_x|\\
    \ge & |X|\gamma_{\rm sp}(H)+(n-|X|)n'\\
    = & nn'-|X|(n'-\gamma_{\rm sp}(H))\\
       \ge & nn'-\alpha(C_n)(n'-\gamma_{\rm sp}(H))\\
        =&\left \lfloor  \frac{n}{2}\right\rfloor \gamma_{\rm sp}(H)+n'\left \lceil  \frac{n}{2}\right\rceil,
\end{align*}
as required.
Suppose that $X$ is not independent. Let $S$ be a maximal subset of $X$ which is composed by consecutive vertices of $C_n$. If $S=\{x_i,x_{i+1},x_{i+2},x_{i+3},x_{i+4}\}$, then $$N(x_{i+2})\subseteq \bigcup_{x_j\in S\setminus \{x_{i+2}\}}N(x_{j}),$$ so that we deduce  that $W$ is not a super dominating set of $C_n\circ H$, which is a contradiction. Thus, $|S|\le 4$. We now fix an independent set $S'\subseteq S$ in the following way. If $S=\{x_i,x_{i+1},x_{i+2},x_{i+3}\}$, then $S'=\{x_i,x_{i+3}\}$, if $S=\{x_i,x_{i+1},x_{i+2}\}$, then $S'=\{x_i,x_{i+2}\}$ and, if $S=\{x_i,x_{i+1}\}$, then $S'=\{x_i\}$. Hence, we can construct an independent set $X'\subseteq X$ by replacing every maximal set $S$ defined as above with the corresponding set $S'$. Since $\gamma_{\rm sp}(H)\le  n'-2$,  by  Lemmas \ref{important} and \ref{adjacentcopies} we have that for any $S$ defined as above,
$$\sum_{x\in S}|W_x|=|S|(n'-1)\ge (|S|-|S'|)n'+|S'|\gamma_{\rm sp}(H),$$
which implies that
\begin{align*}
|W|=& \sum_{x\in X'}|W_x|+ \sum_{x\not\in X'}|W_x|\\
    \ge & |X'|\gamma_{\rm sp}(H)+(n-|X'|)n'\\
    = & nn'-|X'|(n'-\gamma_{\rm sp}(H))\\
       \ge & nn'-\alpha(C_n)(n'-\gamma_{\rm sp}(H))\\
        =&\left \lfloor  \frac{n}{2}\right\rfloor \gamma_{\rm sp}(H)+n'\left \lceil  \frac{n}{2}\right\rceil,
\end{align*}

Therefore, 
$\gamma_{\rm sp}(C_n \circ H) = \left \lfloor  \frac{n}{2}\right\rfloor \gamma_{\rm sp}(H)+n'\left \lceil  \frac{n}{2}\right\rceil.$
\end{proof}

For the case of path graphs we have  $\alpha(P_n)= \left \lceil  \frac{n}{2}\right\rceil$ and   $\alpha_2(P_n)=\left \lceil  \frac{2n}{3}\right\rceil$. To complete the study on the super domination of  number of $P_n \circ H$
we need to state the following lemma.

\begin{lemma}
\label{proppath}
Let  $n\ge 4$ be an integer and $V(P_n)=\{v_1,\dots, v_n\}$, where $v_i$ is adjacent to $v_{i+1}$ for any $i\in \{1,\dots , n-1\}$. If $S\subseteq V(P_n)$  and $|S|=\left \lceil  \frac{2n}{3}\right\rceil +1$, then   at least one of the following statements hold.
\begin{itemize}
\item[(a)] There exists a subscript $i\le n-4$ such that  $\{v_i,v_{i+1},\ldots ,v_{i+4}\}\subseteq S$ or $\{v_i,v_{i+1}, v_{i+2},v_{i+4}\}\subseteq S$
or $\{v_i,v_{i+2}, v_{i+3},v_{i+4}\}\subseteq S$.

\item[(b)] $\{v_1,v_2,v_3\} \subseteq S$ or $\{v_{n-2},v_{n-1},v_{n}\}\subseteq S$.
\end{itemize}
\end{lemma}
\begin{proof} Suppose $|S|=\left \lceil  \frac{2n}{3}\right\rceil +1$ and  conditions $(a)$ and $(b)$ do not hold. 
If  $S$ is a $2$-independent set, then  $|S|=\left \lceil  \frac{2n}{3}\right\rceil +1>\alpha_2(P_n)$, which is a contradiction. Otherwise, for any maximal set  $X\subseteq S$ composed by consecutive vertices we have $|X|\le 4$. We consider two cases depending on the cardinality of $X$:
\begin{itemize}  
\item  $X=\{x_i,x_{i+1},x_{i+2}\}$. If $i=2$, then $\{x_{1},x_{5},x_{6}\}\cap S=\emptyset$, if $i=n-3$, then  $\{x_{n},x_{n-4},x_{n-5}\}\cap S=\emptyset$, while in other cases  $\{x_{i-1},x_{i-2},x_{i+3},x_{i+4}\}\cap S=\emptyset$.
\item $X=\{x_i,x_{i+1},x_{i+2},x_{i+3}\}$.  If $i=2$, then $\{x_{1},x_{6},x_{7}\}\cap S=\emptyset$, if $i=n-4$ then $\{x_{n},x_{n-5},x_{n-6}\}\cap S=\emptyset$, while  in other cases $\{x_{i-1},x_{i-2},x_{i+4},x_{i+5}\}\cap S=\emptyset$. 
\end{itemize}

Let $U$ be the set of these maximal sets of cardinality  $|X|=3$ and, analogously, let  $U'$ be the set these maximal sets of cardinality  $|X|=4$. We assume that the labelling in all these sets is induced by the labelling in $P_n$. Next,  we can construct a set $S^*$ from $S$ by removing $x_{i}$ and adding $x_{i-1}$ for each set in  $U$, and by removing $x_{i+1}$ and adding $x_{i-1}$ for each set in  $U'$. Hence,  $S^*$  is a $2$-independent set, which is a contradiction, as $|S^*|=|S|=\left \lceil  \frac{2n}{3}\right\rceil +1>\alpha_2(P_n)$.
\end{proof}

\begin{proposition} \label{PropositionPaths} Let $H$ be a nonempty graph of order $n'$ and  let $n\ge 2$ be an integer. If $\gamma_{\rm sp}(H)= n'-1$ and $H\not \cong K_{n'}$,   then
$$\gamma_{\rm sp}(P_n \circ H) =  nn'-\left \lceil  \frac{2n}{3}\right\rceil.$$
Furthermore,  if $\gamma_{\rm sp}(H)\le  n'-2$, then 
$$\gamma_{\rm sp}(P_n \circ H) =  \left \lceil  \frac{n}{2}\right\rceil \gamma_{\rm sp}(H)+n'\left \lfloor  \frac{n}{2}\right\rfloor.$$ 
\end{proposition}

\begin{proof} Let $V(P_n)=\{x_1,x_2,\dots , x_n\}$, where $x_i$ is adjacent to $x_{i+1}$ for every $i\in \{1,\dots , n-1\}$.  Let $W$ be a $\gamma_{\rm sp}(P_n\circ H)$-set and $$X=\{x\in V(P_n):\, |W_x|<n'\}.$$

We first consider the case  $\gamma_{\rm sp}(H)= n'-1$ and $H\not \cong K_{n'}$. By Theorem \ref{MainTheorem} we have that $\gamma_{\rm sp}(P_n \circ H) \le  nn'-\left \lceil  \frac{2n}{3}\right\rceil.$ 
It remains to show that $\gamma_{\rm sp}(P_n \circ H) \ge  nn'-\left \lceil  \frac{2n}{3}\right\rceil.$

The cases $n=2$ and $n=3$ were previously discussed in Propositions \ref{G-Complete} and \ref{Bipartite}. From now on we assume that $n\ge 4$. Suppose that $\gamma_{\rm sp}(P_n \circ H) < nn'-\left \lceil  \frac{2n}{3}\right\rceil$.  
Since $\gamma_{\rm sp}(H)= n'-1$, from Lemma~\ref{important} we know that $|\overline{W}_x|=1$ for every $x\in X$, so that $|\overline{W}|=|X|\geq \left \lceil  \frac{2n}{3}\right\rceil+1$.
By Lemma \ref{proppath} we differentiate the following cases.

\begin{itemize}
\item[(a)] There exists a subscript  $i\le n-4$ such that  $X_0=\{v_i,v_{i+1},\ldots ,v_{i+4}\}\subseteq X$ or $X_1=\{v_i,v_{i+2}, v_{i+3},v_{i+4}\}\subseteq X$
or $X_2=\{v_i,v_{i+1}, v_{i+2},v_{i+4}\}\subseteq X$. For any $l\in \{0,1,2\}$ we have that $$N(v_{i+2})\subseteq \bigcup_{v_j\in X_l\setminus \{v_{i+2}\}}N(v_j).$$

\item[(b)] $\{v_1,v_2,v_3\} \subseteq X$ or $\{v_{n-2},v_{n-1},v_{n}\}\subseteq X$. Thus,   $N(v_1)\subseteq N(v_2)\cup N(v_3)$ or $N(v_n)\subseteq N(v_{n-1})\cup N(v_{n-2})$.
\end{itemize}
According to the two cases above we conclude that 
 $W$ is not a super dominating set of $P_n\circ H$, which is a contradiction.  Therefore, $\gamma_{\rm sp}(P_n \circ H) =  nn'-\left \lceil  \frac{2n}{3}\right\rceil.$

From now on we assume that $\gamma_{\rm sp}(H)\le  n'-2$. By Theorem \ref{MainTheorem} we have that $\gamma_{\rm sp}(P_n \circ H) \le   \left \lceil  \frac{n}{2}\right\rceil \gamma_{\rm sp}(H)+n'\left \lfloor  \frac{n}{2}\right\rfloor.$ We will  show that 
$\gamma_{\rm sp}(P_n \circ H) \ge  \left \lceil  \frac{n}{2}\right\rceil \gamma_{\rm sp}(H)+n'\left \lfloor  \frac{n}{2}\right\rfloor.$
Although  this part of the proof is completely analogous to the second part of the proof of Proposition \ref{PropositionCycles}, we prefer to include it for completeness. 
If $X$ is an independent set, then by Lemma \ref{important},
\begin{align*}
|W|=& \sum_{x\in X}|W_x|+ \sum_{x\not\in X}|W_x|\\
    \ge & |X|\gamma_{\rm sp}(H)+(n-|X|)n'\\
    = & nn'-|X|(n'-\gamma_{\rm sp}(H))\\
       \ge & nn'-\alpha(P_n)(n'-\gamma_{\rm sp}(H))\\
        =&\left \lceil  \frac{n}{2}\right\rceil \gamma_{\rm sp}(H)+n'\left \lfloor  \frac{n}{2}\right\rfloor,
\end{align*}
as required.
Suppose that $X$ is not independent. Let $S$ be a maximal subset of $X$ which is composed by consecutive vertices of $P_n$. If $S=\{x_i,x_{i+1},x_{i+2},x_{i+3},x_{i+4}\}$, then $N(x_{i+2})\subseteq \cup_{j=1}^4N(x_{j})$, so that we deduce  that $W$ is not a super dominating set of $P_n\circ H$, which is a contradiction. Thus, $|S|\le 4$. We now fix an independent set $S'\subseteq S$ in the following way. If $S=\{x_i,x_{i+1},x_{i+2},x_{i+3}\}$, then $S'=\{x_i,x_{i+3}\}$, if $S=\{x_i,x_{i+1},x_{i+2}\}$, then $S'=\{x_i,x_{i+2}\}$ and, if $S=\{x_i,x_{i+1}\}$, then $S'=\{x_i\}$. Hence, we can construct an independent set $X'\subseteq X$ by replacing every maximal set $S$ defined as above with the corresponding set $S'$. Since $\gamma_{\rm sp}(H)\le  n'-2$,  by  Lemmas \ref{important} and \ref{adjacentcopies} we have that for any $S$ defined as above,
$$\sum_{x\in S}|W_x|=|S|(n'-1)\ge (|S|-|S'|)n'+|S'|\gamma_{\rm sp}(H),$$
which implies that
\begin{align*}
|W|=& \sum_{x\in X'}|W_x|+ \sum_{x\not\in X'}|W_x|\\
    \ge & |X'|\gamma_{\rm sp}(H)+(n-|X'|)n'\\
    = & nn'-|X'|(n'-\gamma_{\rm sp}(H))\\
       \ge & nn'-\alpha(P_n)(n'-\gamma_{\rm sp}(H))\\
        =&\left \lceil  \frac{n}{2}\right\rceil \gamma_{\rm sp}(H)+n'\left \lfloor  \frac{n}{2}\right\rfloor,
\end{align*}

Therefore, 
$\gamma_{\rm sp}(P_n \circ H) =\left \lceil  \frac{n}{2}\right\rceil \gamma_{\rm sp}(H)+n'\left \lfloor  \frac{n}{2}\right\rfloor.$
\end{proof}

\subsection{Super domination in join graphs }\label{SubsectionJoin}

Since $K_n+K_{n'}=K_{n+n'}$ and $N_n+N_{n'}=K_{n,n'}$, in this section we consider the case of join graphs $G+H$ where $G$ and $H$ are not simultaneously complete nor empty.

Given a graph $G$, a set  $X\subseteq V(G)$ and  a vertex $y\in \overline{X}$, we denote the set of external neighbours of $y$ with respect to $\overline{X}$ by
$$F_X(y)=\{x\in X: \, N(x)\cap \overline{X}=\{y\}\}.$$

\begin{theorem}\label{MainTheoremJoin}
Let  $G$ and $H$ be two nonempty  and noncomplete graphs of order $n$ and  $n'$, respectively. Then
$$\gamma_{\rm sp}(G+H)=\min\{n+n'-2,n+\gamma_{\rm sp}(H),n'+\gamma_{\rm sp}(G)\}.$$
\end{theorem}
\begin{proof}
Let $S_1$ be a $\gamma_{\rm sp}(G)$-set and $S_2$  a $\gamma_{\rm sp}(H)$-set. Let  $g\in V(G)$ and $h\in V(H)$ be nonuniversal vertices of $G$ and $H$, respectively. It is readily seen that $V(G)\cup S_2$, $V(H)\cup S_1$ and $(V(G)\cup V(H))\setminus \{g,h\}$ are super dominating sets of $G+H$, so that 
\begin{equation}\label{UperBoundJoin}
\gamma_{\rm sp}(G+H)\le \min\{n+n'-2,n+\gamma_{\rm sp}(H),n'+\gamma_{\rm sp}(G)\}.
\end{equation}

Now we take a   $\gamma_{\rm sp}(G+H)$-set  $W$ and differentiate the following three cases.

\noindent Case 1. $V(G)\cap \overline{W}\ne \emptyset$ and $V(H)\cap \overline{W} \ne \emptyset$. If $g\in V(G)\cap \overline{W}$ and $h\in V(H)\cap \overline{W}$, then $F_W(g)\subseteq V(H)$ and  $F_W(h)\subseteq V(G)$,   which implies that $\overline{W}=\{g,h\}$. Hence,  $\gamma_{\rm sp}(G+H)= n+n'-2.$ 

\noindent Case 2. $\overline{W} \subseteq V(G).$ In this case, by analogy to the proof of Lemma \ref{important} we deduce that 
$|W\cap V(G)|\ge \gamma_{\rm sp}(G)$, which implies that  $|W|\ge n' +\gamma_{\rm sp}(G)$ and  by \eqref{UperBoundJoin} we deduce that
$\gamma_{\rm sp}(G+H)= n'+\gamma_{\rm sp}(G).$

\noindent Case 3. $\overline{W} \subseteq V(H).$
This case is analogous to the previous one, 
so that
$\gamma_{\rm sp}(G+H)= n+\gamma_{\rm sp}(H).$

According to the three cases above, the result follows.
\end{proof}

Since $K_n+N_{n'}\in \mathcal{F}$, by Theorem \ref{familyF} we have that $\gamma_{\rm sp}(K_n+N_{n'})=n+n'-1$. Hence, it remains to study the cases $K_n+H$ and $N_n+H$ where $H\not \in \{N_{n'},K_{n'}\}$.

\begin{theorem}
Let $H$ be a graph of order $n'$. If $H\not \in \{N_{n'},K_{n'}\}$, then for any integer $n\ge 1$,
$$ \gamma_{\rm sp}(K_n+H)=n+\gamma_{\rm sp}(H).$$
\end{theorem}

\begin{proof}
Let $S$ be a $\gamma_{\rm sp}(H)$-set. It is readily seen that $V(K_n)\cup S$ is a super dominating sets of $K_n+H$, so that 
\begin{equation}\label{UperBoundKn+H}
\gamma_{\rm sp}(K_n+H)\le n+\gamma_{\rm sp}(H)\le n+n'-1.
\end{equation}

Now, let $W$ be a   $\gamma_{\rm sp}(K_n+H)$-set. Since the vertices in $V(K_n)$ are universal vertices of $K_n+H$, 
 if $V(K_n)\cap \overline{W}\ne \emptyset$, then $V(H)\cap \overline{W} = \emptyset$ and, in such a case, $\gamma_{\rm sp}(K_n+H)= n+n'-1$, so that \eqref{UperBoundKn+H} leads to $
  \gamma_{\rm sp}(K_n+H)= n+\gamma_{\rm sp}(H)$.
 On the other hand, if   $\overline{W} \subseteq V(H)$, then by analogy to the proof of Lemma \ref{important} we deduce that 
$|W\cap V(H)|\ge \gamma_{\rm sp}(H)$, which implies that  $|W|\ge n +\gamma_{\rm sp}(H)$ and  by \eqref{UperBoundKn+H} we deduce that
$\gamma_{\rm sp}(K_n+H)= n+\gamma_{\rm sp}(H).$
Therefore, the result follows.
\end{proof}

\begin{theorem}
Let $H$ be a graph of order $n'$. If $H\not \in \{N_{n'},K_{n'}\}$, then for any integer $n\ge 2$,
$$ \gamma_{\rm sp}(N_n+H)=\min\{n'+n-2,n+\gamma_{\rm sp}(H)\}.$$
\end{theorem}

\begin{proof}
Let $W$ be a $\gamma_{\rm sp}(N_n+H)$-set. Notice that $|\overline{W}\cap V(N_n)|\le 1$. Since $H\not \in \{N_{n'},K_{n'}\}$, by  Theorem \ref{familyF} we deduce that 
$ \gamma_{\rm sp}(N_n+H)\le n'+n-2$, which implies that $\overline{W}\cap V(H)\ne \emptyset$. With this fact in mind, and following a procedure analogous to that in the proof of Theorem \ref{MainTheoremJoin}, we conclude the proof.
\end{proof}

\end{document}